\documentclass[twoside,11pt]{amsart}

\usepackage{fullpage}
\usepackage{amssymb,amsmath,amsfonts}
\usepackage{manfnt} 
\usepackage{epsfig,graphicx,psfrag}
\usepackage[english]{babel}
\usepackage[dvips]{hyperref}

\usepackage{color}

\hypersetup{
  pdftitle={},
  pdfauthor={},
  pdfsubject={},
  pdfkeywords={},
  colorlinks=false,
  breaklinks=true,
  bookmarksopen=true,
  bookmarksnumbered=true,
  pdfpagemode=UseOutlines,
  plainpages=false}

\newtheorem{theorem}{Theorem}[section]

\newtheorem{corollary}[theorem]{Corollary}
\newtheorem{definition}[theorem]{Definition}

\newtheorem{lemma}[theorem]{Lemma}
\newtheorem{proposition}[theorem]{Proposition}

\newenvironment{remark}{\vspace{4pt}\noindent\textbf{Remark.}}{\vspace{4pt}}

\newcommand{\IC}{\ensuremath{\mathbb{C}}}

\newcommand{\IZ}{\ensuremath{\mathbb{Z}}}

\newcommand{\cP}{\ensuremath{\mathcal{P}}}
\newcommand{\cl}{\ensuremath{\ell}}

\newcommand{\fp}{\ensuremath{\mathfrak{p}}}

\DeclareMathOperator{\horb}{H_\mathrm{orb}^{*,*}}
\DeclareMathOperator{\codim}{codim}
\newcommand{\nth}{\ensuremath{{}^\mathrm{th}\;}}

\numberwithin{equation}{section}
\begin{document}
\title[Borcea-Voisin construction]{Generalized Borcea-Voisin Construction}
\author{Jimmy Dillies}
\address{University of Utah\\
Department of Mathematics\\
Salt Lake City, UT\\} 
\email{dillies@math.utah.edu}

\thanks{The author would like to thank Professor Ron Donagi for informing him of the existence of references \cite{BD96} and \cite{Rei02}, Max Pumperla 
for pointing out a typo in an earlier version and the anonymous referees for thoughtful suggestions.}

\begin{abstract}
C. Voisin and C. Borcea have constructed mirror pairs of families of Calabi-Yau threefolds by taking the quotient of the product of an elliptic curve with a K3 surface endowed with a non-symplectic involution. In this paper, we generalize the construction of Borcea and Voisin to any prime order and build three and four dimensional Calabi-Yau orbifolds. We classify the topological types that are obtained and show that, in dimension 4, orbifolds built with an involution admit a crepant resolution and come in topological mirror pairs. We show that for odd primes, there are generically no minimal resolutions and the mirror pairing is lost.
\end{abstract}



\maketitle

\setcounter{tocdepth}{1} {\scriptsize \tableofcontents}

\section{Introduction}

The first family of mirror varieties which were neither toric, nor complete intersections, was introduced independently by Borcea \cite{Bor92} and Voisin \cite{Voi93}. 
They construct Calabi-Yau threefolds and fourfolds by taking the quotient by an appropriate involution of the product of lower dimensional Calabi-Yau varieties (with $\IZ/2\IZ$ symmetry, or larger symmetry in the case of Borcea). Their construction and the closedness under ther mirror map of the manifolds they build, rely on the \emph{duality} between K3 surfaces endowed with a non-symplectic involution, which was discovered by Nikulin \cite{Ni81}. \\
The Borcea-Voisin construction is actually similar to the one of Vafa and Witten \cite{VW95} who take the quotient of the product of three tori by a group of automorphisms which preserve the volume form {to study interesting physical models}. In both cases, the procedure consists of taking as building blocks varieties $X_1, \ldots, X_n$, each endowed with a unique volume form,
 and then taking the quotient of their product by some subgroup of the product of automorphism groups consisting of symmetries preserving the total volume form. This approach has given rise to semi-realistic heterotic models whose study was pioneered by Dixon, Harvey, Vafa and Witten in \cite{DHVW}. All possible varieties obtained as quotients of products of tori were classified by Donagi and Faraggi \cite{DoFa}, Donagi and Wendland \cite{DoWe}, and Dillies \cite{Di07}.\\
In this paper, we study and classify the Calabi-Yaus obtained through the construction of Borcea and Voisin in two directions. First, we allow cyclic groups of arbitrary prime order, that is, any prime less than or equal to $19$, as shown in \cite{Ni81}. Second, we construct both three and fourfolds. Our aim is, first, to find new mirror manifolds, second, to illustrate what could go wrong in constructing mirror maps when using higher order automorphisms.

Note that some of the examples which appear in our classification were known before. For example, Borcea \cite{Bor92,Bor97} constructs several examples of Calabi-Yau varieties in higher dimension.
In \cite{MM00}, Mitsuko and Masamichi consider threefolds and fourfolds which can be studied through toric geometry. They obtain a small subset of the classification which we obtain hereunder.
Cynk and Hulek \cite{CH07} construct and study examples of threefolds and fourfolds using involutions and higher order automorphisms and prove modularity for classes with 
complex multiplication. In \cite{Roh10}, Rohde studies, among others, Calabi-Yau threefolds obtained by taking the quotient of the product of an elliptic curve and a K3 surface by
an automorphism of order $3$ fixing only points or rational curves on the K3 surface. Garbagnati and van Geemen \cite{GvG10} study explicitly the Picard-Fuchs equation of a certain 
familiy constructed by Rohde. Finally, in \cite{Gar10}, Garbagnati constructs examples of Calabi-Yau threefolds by using automorphisms of order $4$.

\section{Plan}

After giving a brief description of the notation in Section \ref{sec:not}, we will start by an overview of the results in Section \ref{sec:res}.
In Section \ref{sec:con}, we describe the generalized construction of Borcea-Voisin orbifolds. In Section \ref{sec:orbi}, we compute the orbifold cohomology of our varieties and in Section \ref{sec:top} we list all possible topological types of Borcea-Voisin spaces which we obtain and describe their fundamental group. Finally, in Section \ref{sec:mir}, we discuss how much of the mirror map established for Borcea-Voisin threefolds in \cite{Bor97,Voi93} remains true for our construction.\\
The appendix synthesizes the results about non-symplectic automorphisms of prime order on K3 surfaces which are needed for the classification.

\section{Notation}\label{sec:not}

In this paper, we will denote by $X$ the product of the Calabi-Yau manifolds $X_1$ and $X_2$. The surface $X_i$ is endowed with a non-symplectic automorphism $\rho_i$ of order $\fp$ that fixes a curve of genus $g$, $l_{i}$ rational curves, and $p_{i}$ isolated points.\footnote{ {\tiny\textdbend} When a K3 surface has no fixed points, we will set $g=0$ and, following the convention of \cite{AST09}, $l=-1$.}. 
These points will be further characterized by the linearized form of $\rho$ along their tangent space. We will call $n_i$ the number of points of type $\frac1\fp (i+1,\fp-i)$.
Each surface $X_i$ will be characterized by a triple $(\fp,r_i,a_i)$ where $\fp$ is the order of the non-symplectic automorphism $\rho_i$,  $r_i$ the rank of the invariant part $S$ of $H^2(X_i)$ and $\det (S) =\fp ^{a_i}$.
 Moreover, $\lambda_i$ will be a shorthand notation for $\frac{22-r_i}{\fp-1}$. 
Also, at times, we will use the symbol $\alpha$ as defined in \cite{AST09} or in Table \ref{t:ast}.

\section{Results}\label{sec:res}

Using the recent classification of non-symplectic automorphisms of prime order on K3 surfaces (see \cite{AST09}, or Appendix \ref{app:cas} for a working synopsis) we construct families of generalized Calabi-Yau orbifolds in dimensions 3 and 4. We classify all topological families and see that for $\fp$ odd, distinct pairs of K3 surfaces yield distinct Calabi-Yau orbifolds.\\
 In dimension 4, the family attached to the prime $2$ consists of orbifolds admitting crepant resolutions and the set of Calabi-Yau manifolds which we obtain is closed under the topological mirror map, i.e. for each $X$ in our family, there exists another variety $\check{X}$ such that $h^{p,q}(X)=h^{d-p,q}(\check{X})$.\\
For the other primes, we show that except in two cases, there is no crepant resolution, and that these Calabi-Yau orbifolds do not come in mirror pairs.

\section{Construction}\label{sec:con}

Consider two pairs $(X_1,\rho_1)$ and $(X_2,\rho_2)$ each consisting of a Calabi-Yau manifold and a primitive non-symplectic automorphism of order $\fp$. Pick a volume form $\omega_i$ on each surface. Without loss of generality, we can assume that the characters induced by the action on the volume forms are identical, i.e. $\rho_1^* \omega_1 = \zeta_p \omega_1$ and $\rho_2^* \omega_2 = \zeta_p \omega_2$, for the same $\fp\nth$ root of unity $\zeta_\fp$. 
On the product variety $X=X_1\times X_2$, we get an induced action of $\IZ/\fp\IZ \times \IZ/\fp\IZ$ and a character map
$$ \chi : \IZ/\fp\IZ \times \IZ/\fp\IZ \rightarrow \IC^\times : (\rho_1^i, \rho_2^j) \mapsto \zeta_\fp^{i+j}$$
which determines the action on the volume form of $X$. Let $G\cong \IZ/\fp\IZ$ be the kernel of $\chi$. 

\begin{definition}
Given two pairs $(X_1,\rho_1)$ and $(X_2,\rho_2)$ as above, we call $X/G$ the associated Borcea-Voisin orbifold. If a crepant resolution $\widetilde{X/G}$ of the quotient exists, we call it the associated Borcea-Voisin manifold.
\end{definition}

The case where $\fp=2$, $\dim X_1=1$ and $\dim X_2=2$ is the classical Borcea-Voisin construction.\\
Note that since $G$ consists of automorphisms preserving the volume form of the product, the quotient still has a volume form. Moreover, all singularities are Gorenstein, thus, by the work of Roan \cite{Ro96}, we can conclude that if the codimension of each component of the fixed locus of $(\rho_1,\rho_2)$ has codimension strictly less than $4$, then such a resolution exists. In particular, this is the case when $\dim X \leq 3$.

\section{Orbifold cohomology}\label{sec:orbi}

We start from the orbifold cohomology formula of Chen and Ruan \cite{CR04}
\begin{equation}
\label{eq:orbi}
\horb \left(X/G\right) = \bigoplus_{g\in \mathrm{Conj}(G)} 
\bigoplus_{\Lambda \in \Phi(g)} H^{*-\kappa(g,\Lambda),*-\kappa(g,\Lambda)}\left(\Lambda\right)^G 
\end{equation}
where $\Phi(g)$ is the set of irreducible components fixed by $g$, and $\kappa(g,\Lambda)$ is the age of $g$ at a point of $\Lambda$. Since the group $G$ is cyclic of prime order, we can pick a generator $\gamma$ and the formula simplifies to
$$\horb \left(X/G\right) = H^{*,*}\left(X\right)^G \oplus
\bigoplus_{\Lambda \in \Phi(\gamma)} \bigoplus_{i=1}^{\fp-1} H^{*-\kappa(\gamma^i,\Lambda),*-\kappa(\gamma^i,\Lambda)}\left(\Lambda\right) .$$
To linearize the notation we will make use of Hodge polynomials, i.e. $h(X)(s,t)$ is an integral polynomial whose coefficient of bi-degree $(i,j)$ is the dimension $\dim H^{i,j}(X)$. 
Hence, what we need to determine, for each orbifold $X$, is the following Hodge polynomial:
\begin{equation}
\label{eq:horb}
h_\mathrm{orb}(X/G)(s,t)=h(X)^G (s,t)+ \sum_{\Lambda \in \Phi(\gamma)} \sum_{i=1}^{\fp-1}  (st)^{\kappa(\gamma^i,\Lambda)}h\left(\Lambda\right)(s,t) .
\end{equation}
In the following sections we will compute individually each summand. We will start by the invariant part and then focus on the contribution coming from the fixed components together with their associated weight $(st)^{\kappa(\gamma^i,\Lambda)}$. Fixed components will be separated according to their codimension.

\subsection{G-invariant part}

\begin{lemma}
\begin{itemize}
Given p, assume that the K3 surface $X_i$ is defined by the pair $(r_i,a_i)$ and let $\lambda_i=\frac{22-r_i}{\fp-1}$. If $X_i$ is an elliptic curve, $\fp\in\{2,3\}$ and $X_i$ is unique.
\item When both $X_i$'s are K3 surfaces, the $G$-invariant part, $h(X)^G(s,t)$, decomposes as 
\begin{equation}
\label{eq:ginv4}
1+(st)^4+s^4+t^4+
(st^3+ts^3)(\lambda_1+\lambda_2-2)+(st+(st)^3)(r_1+r_2)+
(st)^2(4-2(\lambda_1+\lambda_2)+r_1 r_2 + (\fp-1) \lambda_1 \lambda_2)
\end{equation}
when $\fp$ is odd, and as
\begin{equation}
\label{eq:ginv42}
1+(st)^4+s^4+t^4+
(st^3+ts^3)(\lambda_1+\lambda_2-4)+(st+(st)^3)(r_1+r_2)+
(st)^2(8-2(\lambda_1+\lambda_2)+r_1 r_2 +  \lambda_1 \lambda_2)
\end{equation}
when $\fp=2$.
\item When $X_1$ is a K3 surface and $X_2$ is an elliptic curve,  the $G$-invariant part, $h(X)^G(s,t)$, decomposes as 
\begin{equation}
\label{eq:ginv3}
1+(st)^3+s^3+t^3+(r_1+1)(st+(st)^2)+(\lambda_1-1)(st^2+s^2t) .
\end{equation}
\end{itemize}
\end{lemma}

\begin{proof}
By K\"unneth's formula one can decompose the $G$-invariant cohomology of $X$ as
$$h(X)^G (s,t)= \sum_{i=0}^{\fp-1} h(X_1)[\zeta_\fp^i](s,t) \times h(X_2)[\zeta_\fp^i](s,t)$$
where $H(X_i)[\zeta]$ is the $\zeta$ eigenspace associated to the action of $G_i=\IZ/\fp\IZ$, generated by $\rho_i$. 
When $X_i$ is a K3 surface, we define $\lambda_{i,\cl}=\dim (H^{1,1}(X_i) \cap H(X_i)[\zeta_\fp^\cl])$. Since the trace of $\zeta^*$ is integral, we conclude that for $\cl>0$, $\lambda_{i,1}=\cdots=\lambda_{i,{\fp-1}}=\lambda_i=\frac{22-r_i}{\fp-1}$. Therefore, if $\fp>2$,
\begin{enumerate}
\item $h(X_i)[1](s,t)=(1+(st)^2)+ r_i st$, 
\item $h(X_i)[\zeta](s,t)=h(X_i)[\zeta^{\fp-1}](t,s)=s^2+(\lambda_i-1)st$, and
\item $h(X_i)[\zeta^\cl](s,t)=\lambda_i st$ for $2\leq \cl \leq p-2$
\end{enumerate}
while, if $\fp=2$ 
\begin{enumerate}
\item $h(X_i)[1](s,t)=(1+(st)^2)+ r_i st$, 
\item $h(X_i)[-1](s,t)=h(X_i)[\zeta^{\fp-1}](t,s)=s^2+(\lambda_i-2)st+t^2 .$
\end{enumerate}
When $X_i$ is an elliptic curve, we only need to deal with $\fp=3$ as this is the only possible prime order of complex multiplication besides $2$. In that case,
\begin{enumerate}
\item $h(X_i)[1](s,t)=1+st$, 
\item $h(X_i)[\zeta](s,t)=h(X_i)[\zeta^{\fp-1}](t,s)=s$. 
\end{enumerate}

The final result follows from a direct computation.
\end{proof}

\subsection{Contribution of fixed locus}

Since the codimension of the fixed locus contributes to a difference in grading in the contribution to cohomology, we will split the fixed locus accordingly. For threefolds, we will have codimension 2 and 3 fixed loci and for fourfolds, we will have fixed loci of codimension 2,3 or 4. Before we proceed with our analysis, we will begin by a short study of the possible ages along the fixed components.

\subsubsection{Preamble on age}

Let $\rho$ be an automorphism of $X$ which fixes a point $P$. If we diagonalize the linearization
of the action of $\rho$ at $P$ as 
$$\begin{pmatrix}
e^{2\pi i k_1} & 0 & 0\\
 & \ddots & \\
0 & 0 & e^{2\pi i k_n}
\end{pmatrix}$$
with $0\leq k_i <1$, then we define the age of $\rho$ along the fixed component $\Lambda$ containing $P$ as
$$\kappa(\rho,\Lambda)=\sum_{i=1}^n k_i .$$
\begin{lemma}
\label{lem:age}
The age of an element and its inverse at a given point are related by $\kappa(\rho,\Lambda)+\kappa(\rho^{-1},\Lambda)=\codim(\Lambda,X)$.
\end{lemma}

\begin{proof}
If $\rho$ linearizes to $\mathrm{diag}(1,\ldots,1,e^{2\pi i k_{j+1}},\ldots,e^{2\pi i k_n})$, then, with respect to the same basis, $\rho^{-1}$ takes the form $\mathrm{diag}(1,\ldots,1,e^{2\pi i (1-k_{j+1})},\ldots,e^{2\pi i (1-k_n)})$, where $j=\dim(\Lambda)$ and $k_l >0$ for $l>j$.
\end{proof}

\begin{corollary}
\label{cor:age}
Let $G$ be a group acting faithfully on a manifold $X$ of dimension $n\geq 3$, and whose tangent action at a fixed point lies in the special linear group. Take $g$ an element different from the identity, and $\Lambda$ an irreducible component of the fixed locus of $g$. We have:
\begin{enumerate}
 \item $\kappa(g,\Lambda)=1$ if $\Lambda$ is of codimension $2$.
 \item $\kappa(g,\Lambda)\in \{1,2\}$ if $\Lambda$ is of codimension $3$.
 \item  $\kappa(g,\Lambda)\in \{1,2,3\}$ if $\Lambda$ is of codimension $4$.
\end{enumerate}
\end{corollary}

\begin{proof}
Note that when $g\in G$ is a finite subgroup of $SL$, the age of $g$ is an integer. 
Since $g$ is different from the identity and the action is faithful, $\kappa(g,.)\geq 1$. The rest follows from the previous lemma.
\end{proof}

\subsection{Threefolds}

Given that the only possible prime orders for complex multiplication are 2 and 3, in this section we will only  make use of K3 surfaces endowed with non-symplectic automorphisms of order 3. Also recall that an elliptic curve with complex multiplication of order 3 is defined uniquely. Moreover, the automorphism fixes 3 isolated points.

\subsubsection{Codimension 2}

\begin{lemma}
The contribution to the cohomology from the fixed curves on $X$ is
\begin{equation}
\label{eq:co3}
h(\mathcal{S})(s,t)=3\left((1+st)(l_1 +1)+(s+t)g_1 \right)
\end{equation}
\end{lemma}

\begin{proof}
We know from \cite{AS08} that for $\fp=3$ there is at most one curve of non-zero genus. Recalling that there are $3$ fixed points on the elliptic curve, the result is immediate.
\end{proof}

From Lemma \ref{lem:age}, we know that the age of any element along a fixed curve is 1 and this polynomial is thus weighted by $(\fp-1)st$.

\subsubsection{Codimension 3}

In this case, we are dealing with points and the Hodge polynomial for each fixed point is equal to 1. Counting them all, we get a contribution of $$h(\mathcal{P})(s,t)=3 p_1 .$$ 
Moreover, since we are in codimension 3, Lemma \ref{lem:age} tells us that if an element has age 1 at a point $P$, its inverse will have age 2. The weight of the contribution of the fixed points is thus
$$(\fp-1)\frac{st+(st)^2}{2} .$$ 
\subsection{Fourfolds}

\subsubsection{Codimension 2}

The dimension $2$ part of the fixed locus on $X$ consists of the product of the fixed curves on each $X_i$. 
The fixed curves on $X_i$ have as total Hodge polynomial $h(\mathcal{C}_i)(s,t)=(1+st)(l_i +1)+(s+t)(g_i)$. From K\"unneth's formula we get directly:

\begin{lemma}
The contribution to the cohomology from the fixed surfaces on $X$ is 
\begin{equation}
\label{eq:co32}
h(\mathcal{S})(s,t)=(1+(st)^2)(l_1 +1)(l_2 +1)+(s+t+s^2t+t^2s)[(1+l_1)g_2+(1+l_2)g_1] 
+2st[(1+l_1)(1+l_2)+g_1 g_2].
\end{equation}
\end{lemma}

Moreover, we know from Corollary \ref{cor:age} that $\kappa(.,\mathcal{S})$ is identically equal to $1$ and thus the above polynomial is weighted by $(\fp-1)(st)$. 

\subsubsection{Codimension 3}

\begin{lemma}
The contribution to the cohomology from the fixed curves on $X$ is
\begin{equation}
\label{eq:co33}
h(\mathcal{S})(s,t)=(1+st)\left((l_1 +1)p_2+(l_2+1)p_1\right)+(s+t)\left(g_1 p_2 + g_2 p_1 \right) .
\end{equation}
\end{lemma}

\begin{proof}
The result follows once again from K\"unneth's formula:
$$h(\mathcal{C})(s,t)=\sum_{i=1}^2 h(\mathcal{C}_i)(s,t)\times h(\mathcal{P}_{3-i})(s,t).$$
\end{proof}

Since $\kappa(g,\mathcal{C})=3-\kappa(g^{-1},\mathcal{C})$ (Corollary \ref{cor:age}), the above polynomial is weighted by $(\fp-1)\frac{(st)+(st)^2}{2}$.

\subsubsection{Codimension 4}

In this case, we are dealing with points and the Hodge polynomial for each fixed point is trivial: $$h(\mathcal{P})(s,t)=1 .$$
On the other hand, it is not automatic to deduce the weight coming from the age of the powers of the generator at a given point. In this section, given a variety $X$, we will determine the integers $N_i(\fp)$, which are the number of times a power of the group generator has an age of $i\in \{1, 2, 3\}$ at an isolated fixed point. The contribution to the cohomology will then be of 
\begin{equation}
\label{eq:co4} 
\sum_{i=1}^3 N_i(\fp) (st)^i .
\end{equation}
From Lemma \ref{lem:age} and Corollary \ref{cor:age}, we know that we only need to determine the number of times a point comes with a shift of $2$.  Indeed, we have the relations 
$$N_1(\fp)=N_3(\fp)$$ and
$$\sum_{i=1}^3 N_i(\fp) = (\fp-1) \times \# \{ \textrm{fixed points} \} .$$

\begin{lemma}
The number of shifts of $2$ given by powers of a generator $\gamma$ at a fixed point of $X$ of type $\frac1\fp (q_1+1,-q_1) \times \frac1\fp (q_2+1, -q_2)$ is given by the entry   $(q_1,q_2)$ of the matrix $\cP_2(\fp)$:
$${\cP}_2(3)=
\begin{pmatrix}
2
\end{pmatrix},
\qquad
{\cP}_{2}(5)=
\begin{pmatrix}
4 & 2 \\
2 & 4
\end{pmatrix},
\qquad
{\cP}_2(7)=
\begin{pmatrix}
6 & 4 & 4 \\
4 & 6 & 2 \\
4 & 2 & 6
\end{pmatrix}
$$
$$ \cP_2(11)=
\begin{pmatrix}
10 & 6 & 8 & 8 & 6 \\
6 & 10 & 4 & 8 & 6 \\
8 & 4 & 10 & 6 & 4 \\
8 & 8 & 6 & 10 & 4 \\
6 & 6 & 4 & 4 & 10
\end{pmatrix},
\qquad
\cP_2(13)=
\begin{pmatrix}
12 & 8 & 10 & 8 & 10 & 6 \\
8 & 12 & 6 & 8 & 10 & 6 \\
10 & 6 & 12 & 6 & 8 & 8 \\
8 & 8 & 6 & 12 & 6 & 6 \\
10 & 10 & 8 & 6 & 12 & 4 \\
6 & 6 & 8 & 6 & 4 & 12
\end{pmatrix}$$
$$\cP_2(17)=
\begin{pmatrix}
16 & 10 & 14 & 10 & 12 & 12 & 14 & 8 \\
10 & 16 & 8 & 12 & 10 & 10 & 12 & 10 \\
14 & 8 & 16 & 8 & 10 & 14 & 12 & 6 \\
10 & 12 & 8 & 16 & 6 & 10 & 12 & 10 \\
12 & 10 & 10 & 6 & 16 & 8 & 10 & 8 \\
12 & 10 & 14 & 10 & 8 & 16 & 10 & 8 \\
14 & 12 & 12 & 12 & 10 & 10 & 16 & 6 \\
8 & 10 & 6 & 10 & 8 & 8 & 6 & 16
\end{pmatrix},
$$
$$
 \cP_2(19)=
\begin{pmatrix}
18 & 12 & 14 & 12 & 16 & 12 & 14 & 14 & 10 \\
12 & 18 & 8 & 14 & 14 & 10 & 12 & 16 & 8 \\
14 & 8 & 18 & 8 & 12 & 12 & 14 & 10 & 10 \\
12 & 14 & 8 & 18 & 10 & 10 & 12 & 12 & 12 \\
16 & 14 & 12 & 10 & 18 & 10 & 12 & 16 & 8 \\
12 & 10 & 12 & 10 & 10 & 18 & 8 & 12 & 8 \\
14 & 12 & 14 & 12 & 12 & 8 & 18 & 10 & 10 \\
14 & 16 & 10 & 12 & 16 & 12 & 10 & 18 & 6 \\
10 & 8 & 10 & 12 & 8 & 8 & 10 & 6 & 18
\end{pmatrix}$$
\end{lemma}

\begin{proof}
The computations are long and tedious but trivial. For example, one can notice that along a diagonal, an element will always have a linearization of the form $(n_i+1,\fp-n_i,\fp-n_i-1,n_i)$ and that the age will always be 2. Therefore, any element different from the identity in $G$ will have the same age at such a point.
\end{proof}

\begin{definition}
Given a non-symplectic automorphism of order $\fp$ acting on a K3 surface defined by the pair $(r,a)$, we define $v_\fp(r)$ to be the $\frac{\fp-1}{2}$ dimensional row vector of natural numbers whose $i^\textrm{th}$ entry is the number of fixed points of type $\frac1\fp (i+1,-i)$.
\end{definition}

This vector can be read from Table \ref{t:ast} in the Appendix. Given $\cP_2(\fp)$ and $v_\fp(r_i)$, it is easy to determine the number of times an element acts at a fixed point of $X$ with a given age:

\begin{lemma}
The number of shifts by 2 is equal to
$$N_2(\fp)=v_\fp(r_1) \cP_2(\fp) v_\fp(r_2)^t$$
\end{lemma}

Let $\alpha_i$ be as in Table \ref{t:ast}. We obtain as a direct consequence:

\begin{corollary}
Given a pair of K3 surfaces defined by $(r_1,a_1)$ and $(r_2,a_2)$, we define $\alpha_i$ as in Table \ref{t:ast2}. The number of shifts by $2$ on the product is

\begin{center} 
\begin{table}[ht]
 \begin{tabular}{ccc}
\hline\\[-0.9em]
$\fp$ & $N_2(\fp)$ & $2 N_1(\fp)=2 N_3 (\fp)$ \\
\hline \hline\\[-0.7em]
$3$ & $2(\alpha_1 + 3)(\alpha_2 + 3)$ & $0$\\
$5$ & $28\alpha_1\alpha_2 + 38\alpha_1 + 38\alpha_2 + 52$ & $8\alpha_1\alpha_2 + 10\alpha_1 + 10\alpha_2 + 12$ \\
$7$ & $110\alpha_1 \alpha_2 + 70\alpha_1 + 70\alpha_2 + 46$ & $20\alpha_2 + 8 + 40\alpha_1\alpha_2 + 20\alpha_1$ \\  
$11$ & $36 + 138\alpha_2 + 138\alpha_1 + 570\alpha_1 \alpha_2$ & $4 + 42\alpha_2 + 42\alpha_1 + 240\alpha_1\alpha_2$\\
$13$ & $28 - 154\alpha_2 - 154\alpha_1 + 1012\alpha_1 \alpha_2$ & $20 - 110\alpha_2 - 110\alpha_1 + 440\alpha_1\alpha_2$\\ 
$17$ & $2480\alpha_1 \alpha_2 + 1150\alpha_2 + 1150\alpha_1 + 536$ &  $530\alpha_2 + 1120\alpha_1\alpha_2 + 530\alpha_1 + 248$\\
$19$ & $314 + 1054\alpha_1 + 1054\alpha_2 + 3570\alpha_1 \alpha_2$ & $476\alpha_1 + 476\alpha_2 + 136 + 1632\alpha_1\alpha_2$\\
\hline
\end{tabular}
\vspace{0.5cm}
\caption{Number of isolated points with an action of a given weight.}
\label{t:weight}
\end{table}
\end{center}
\end{corollary}
\section{Topological types}\label{sec:top}
In this section, we give the Hodge diamonds of the generalized Borcea-Voisin orbifolds. We write down a formula for the Euler characteristic and give the range of values that are taken. All possible values that can be taken for $(r_1,a_1)$ and $(r_2,a_2)$ are listed in Appendix \ref{app:cas}.

\subsection{Dimension 3}

\subsubsection{Order 3}
\[ \left(h^{i-1,j-1}(X/G)\right)_{i,j=1\ldots \dim X +1}=
 \left[ 
{\begin{array}{rccr}
1 & 0 & 0 & 1 \\
0 & 7 + 4r - 3a & 43 - 2r - 3a & 0 \\
0 & 43 - 2r - 3a & 7 + 4r - 3a & 0 \\
1 & 0 & 0 & 1
\end{array}}
 \right]
\]
\[\chi=
 - 72 + 12r
\]
\subsection{Dimension 4}

The Hodge diamond of the Borcea-Voisin orbifold has the form
$$\left[{\begin{array}{rrrrr}
1 & 0 & 0 & 0 & 1 \\
0 & a & d & e & 0\\
0 & d & b & d & 0\\
0 & e & d & a & 0\\
1 & 0 & 0 & 0 & 1 
\end{array}}
 \right]$$
where $a,\ldots,f$ are defined below:
\subsubsection{Order 2}
For $\fp=2$, unless $(r_i,a_i)=(10,10)$ for either value of $i$, then
\begin{eqnarray*}
a& =& 1 + {\frac{r_1r_2}{4}}  - 
{\frac{r_1a_2}{4}}  - 
{\frac{a_1r_2}{4}}  + 
{\frac{a_1a_2}{4}}  + 
{\frac{3r_1}{2}}  - { 
\frac{a_1}{2}}  + {\frac{3{r_2}}{2}}  - {\frac{a_2}{2}}   \\
b &= &648 + a_1a_2 - 30r_2 - 30
r_1 - 12a_2 - 12a_1 + 3
r_1r_2  \\
d &= &22 - {\frac{r_1r_2}{2}} 
 + {\frac{a_1a_2}{2}}  + 5r_2 - 6a_2 + 5r_1 - 6{a_1} \\
e &= &161 + {\frac{r_1a_2}{4}}  + {\frac{a_1a_2}{4}}  + {\frac{r_1r_2}{4}}  + 
{\frac{a_1r_2}{4}}  - 
{\frac{13r_2}{2}}  - 
{ \frac{13r_1}{2}}  - 
{\frac{11 a_1}{2}}  - {\frac{11 a_2}{2}}  
\end{eqnarray*}
\[\chi=
888 - 60r_2 - 60r_1 + 6r_1r_2
\]
Moreover, $-92 \leq \chi \leq 888$ and the smallest values of $\chi$ in absolute value are $-6$, $0$ and $18$.

\begin{remark}
For $\fp=2$, we have two special cases which correspond to either $(r_i,a_i)$ being either equal to $(10,8)$ or $(10,10)$. 

\begin{enumerate}
 \item If $(r_1,a_1)=(10,10)$, then $S=U(2)\oplus E_8 (2)$ and the automorphism acts without fixed points. In that case, only the first summand in Formula \ref{eq:horb} contributes to the cohomology, and the Hodge diamond simplifies to 
\[
 \left[ 
{\begin{array}{rrrrr}
1 & 0 & 0 & 0 & 1 \\
0 & 10+r_2 & 0 & 30-r_2 & 0 \\
0 & 0 & 204 & 0 & 0 \\
0 & 30-r_2 & 0 & 10+r_2 & 0 \\
1 & 0 & 0 & 0 & 1
\end{array}}
 \right]
\]
\[\chi=
 288
\]
Note that in this case, when $r_2=10$, the Hodge diamond is completely symmetric and the mirror pairing is preserved as this topological class maps to itself.

 \item If $(r_1,a_1)=(10,8)$, then $S=U\oplus E_8(2)$, and the fixed locus consists of two disjoint genus 1 curves. However, Table \ref{t:ast2} and the genus formula \ref{eq:genus} tell us that the fixed locus consists of a genus 2 curve and an isolated rational curve. This discrepancy between the prediction made by the formula and the actual fixed locus is however totally harmless for our calculations: the total Hodge diamond in both cases is the same
\[
 \left[ 
{\begin{array}{cc}
2 & 2\\
2 & 2
\end{array}}
 \right]
\]
and our final results are not altered.

\end{enumerate}

\end{remark}

\subsubsection{Order 3}

For $\fp=3$,
\begin{eqnarray*}
a &= &{\frac{3}{2}}  + {\frac{9
r_1}{4}}  - a_1 + {\frac{9
r_2}{4}}  - a_2 + {\frac{3
r_1r_2}{8}}  - {\frac{
r_1a_2}{2}}  - {\frac{
a_1r_2}{2}}  + {\frac{
a_1a_2}{2}}\\ 
b &= & 328 - 14r_2 - 13a_2 - 14r_1 - 13a_1 + 3r_1r_2 - {\frac{r_1a_2}{2}}  - 
{\frac{a_1r_2}{2}}  + 2a_1a_2  \\
d &= & 22 + 5r_2 - {\frac{13a_2}{
2}}  + 5r_1 - {\frac{13a_1
}{2}}  - {\frac{r_1r_2}{2}} 
 - {\frac{r_1a_2}{4}}  - 
{\frac{a_1r_2}{4}}  + 
a_1 a_2 \\
e &= &{\frac{a_1r_2}{4}}  + 
{\frac{a_1a_2}{2}}  - 
{\frac{13r_1}{4}}  - { \frac{11a_1}{2}}  - 
{\frac{13r_2}{4}}  - {\frac{11a_2}{2}}  +
 {\frac{r_1a_2}{4}} + {\frac{r_1r_2}{8}}  +
 {\frac{161}{2}} 
\end{eqnarray*}
\[\chi=
408 - 36r_2 - 36r_1 + 6r_1
r_2
\]
Moreover, there are 299 distinct topological families with $-144 \leq \chi \leq 1368$ and the values of $\chi$ which are the closest to $0$ are $-48$, $0$ and $24$. 

\subsubsection{Order 5}

For $\fp=5$,
\begin{eqnarray*}
a &= &{\frac{15}{4}}  + {\frac{25
r_1}{8}}  - 2a_1 + {\frac{25
r_2}{8}}  - 2a_2 + {\frac{
11r_1r_2}{16}}  - r_1{
a_2} - a_1r_2 + a_1{
a_2}  \\
b &= &172 - 4r_2 - 15a_2 - 4r_1
 - 15a_1 + 4r_1r_2 - 
{\frac{3r_1a_2}{2}}  - 
{\frac{3a_1r_2}{2}} + 4a_1a_2\\ 
d &= &22 + 5r_2 - {\frac{15a_2}{2}}  + 5r_1 - {\frac{15
a_1}{2}}  - 
{\frac{r_1r_2}{2}} - 
{\frac{3r_1a_2}{4}}  - 
{\frac{3a_1r_2}{4}}  + 
2a_1a_2\\ 
e &= &{\frac{a_1r_2}{4}}  + a_1a_2 -
 {\frac{13r_1}{8}} - {\frac{11a_1}{2}}  - {\frac{13r_2}{8}}  - { 
 \frac{11a_2}{2}}  + {\frac{{r_1}a_2}{4}}  +
 {\frac{r_1r_2}{16}}  + {\frac{157}{4}} 
\end{eqnarray*}
\[\chi=
174 - 21r_2 - 21r_1 + { 
\frac{15}{2}} r_1r_2
\]
Moreover, there are 28 distinct topological families with $24 \leq \chi \leq 1848$.

\subsubsection{Order 7}

For $\fp=7$,
\begin{eqnarray*}
a &= &{\frac{97}{18}}  + {\frac{139r_1}{36}}  - 3a_1 + {\frac{139r_2}{36}}  - 
3a_2 + {\frac{73r_1r_2}{72}}  - {\frac{3r_1a_2}{2}}  - 
{\frac{3a_1r_2}{2}}  + {\frac{3a_1a_2}{2}} \\
b &= &{\frac{1112}{9}}  + {\frac{10r_2}{9}}  - 17a_2 + 
{ \frac{10r_1}{9}}  - 17a_1 + 
{\frac{47r_1r_2}{9}}  - {\frac{5r_1a_2}{2}}  - 
{\frac{5a_1r_2}{2}}  + 6a_1a_2 \\
d &= &22 + 5r_2  - {\frac{17a_2}{2}}  + 5
r_1 - {\frac{17a_1}{2}}  - 
{\frac{r_1r_2}{2}}  - 
{\frac{5r_1a_2}{4}}  - 
{\frac{5a_1r_2}{4}}  + 3
a_1a_2\\
e &= &{\frac{a_1r_2}{4}}  + 
{\frac{3a_1a_2}{2}}  - 
{\frac{13r_1}{12}}  - { 
\frac{11a_1}{2}}  - {\frac{13
r_2}{12}}  - {\frac{11a_2}{2
}}  + {\frac{r_1a_2}{4}} 
 + {\frac{r_1r_2}{24}
}  + {\frac{51}{2}} 
\end{eqnarray*}
\[\chi=
{\frac{304}{3}}  - {\frac{40}{3}} 
r_2 - {\frac{40}{3}} r_1
 + {\frac{28}{3}} r_1r_2
\]
Moreover, there are 15 distinct topological families $144 \leq \chi \leq 2064$.

\subsubsection{Order 11}

For $\fp=11$,
\begin{eqnarray*}
a & = &{\frac{83}{10}}  + {\frac{21r_1}{4}}  -
 5a_1 + {\frac{21r_2}{4}}  - 5a_2 + 
{\frac{67r_1r_2}{40}}  - {\frac{5r_1a_2}{2}}  -
{\frac{5a_1r_2}{2}}  + {\frac{5a_1a_2}{
2}} \\
b &= &{\frac{456}{5}}  + 
{\frac{42r_2}{5}}  - 21a_2 + {\frac{42r_1}{5}}  - 21a_1 + 
{\frac{39r_1r_2}{5}}  -
{\frac{9r_1a_2}{2}}  - {\frac{9a_1r_2}{2}}  + 10a_1a_2\\
d &= &22 + 5r_2 - {\frac{21a_2}{2}}  + 5
r_1 - {\frac{21a_1}{2}}  - 
{\frac{r_1r_2}{2}}  - 
{\frac{9r_1a_2}{4}}  - 
{\frac{9a_1r_2}{4}}  + 5
a_1a_2 \\
e &= &{\frac{a_1r_2}{4}}  + 
{\frac{5a_1a_2}{2}}  - 
{\frac{13r_1}{20}}  - { 
\frac{11a_1}{2}}  - {\frac{13
r_2}{20}}  - {\frac{11a_2}{2
}}  + {\frac{r_1a_2}{4}} 
 + {\frac{r_1r_2}{40}
}  + {\frac{29}{2}} 
\end{eqnarray*}
\[\chi=
{\frac{264}{5}}  - {\frac{12}{5}} 
r_2 - {\frac{12}{5}} r_1
 + {\frac{66}{5}} r_1r_2
\]
Moreover, there are 6 distinct topological families satisfying $96 \leq \chi \leq 1896$.

\subsubsection{Order 13}
\[
 \left[ 
{\begin{array}{rrrrr}
1 & 0 & 0 & 0 & 1 \\
0 & 404 & 0 & 0 & 0 \\
0 & 0 & 1372 & 0 & 0 \\
0 & 0 & 0 & 404 & 0 \\
1 & 0 & 0 & 0 & 1
\end{array}}
 \right]
\]
\[\chi=
 2184
\]
\subsubsection{Order 17}
\[
 \left[ 
{\begin{array}{rrrrr}
1 & 0 & 0 & 0 & 1 \\
0 & 264 & 0 & 0 & 0 \\
0 & 0 & 844 & 0 & 0 \\
0 & 0 & 0 & 264 & 0 \\
1 & 0 & 0 & 0 & 1
\end{array}}
 \right] 
\]
 \[\chi=
 1376 
 \]
\subsubsection{Order 19}
\[
 \left[ 
{\begin{array}{rrrrr}
1 & 0 & 0 & 0 & 1 \\
0 & 184 & 0 & 0 & 0 \\
0 & 0 & 564 & 0 & 0 \\
0 & 0 & 0 & 184 & 0 \\
1 & 0 & 0 & 0 & 1
\end{array}}
 \right]
\]
 \[\chi=936 
 \]
\subsection{Fundamental groups}

It follows from the argument in \cite{DHVW} that the fundamental group of our orbifolds is isomorphic to the quotient of $G$ by the subgroup generated by elements fixing at least a point. Since $G$ is cyclic of prime order, either the fundamental group will be trivial, or it will be cyclic of order $\fp$. Actually, all orbifolds will be simply connected except for $\fp=2$ when at least one of the surfaces is defined by $(r=10,a=10)$, in which case there are no fixed points and the fundamental group is $\IZ/2\IZ$.

\section{On the mirror map}\label{sec:mir}

Voisin \cite{Voi93} and Borcea \cite{Bor97} show how the Calabi-Yau threefolds which they construct come in topological mirror pairs. 
It is easy to check from the above data that their result extends to fourfolds. Generalized Borcea-Voisin orbifolds with $\fp=2$ admit crepant resolutions and come in topological mirror pairs.

\begin{proposition}
The Borcea-Voisin fourfolds $X$ and $\check{X}$ defined respectively by $(r_1 ,a_1 ,r_2 ,a_2)$ and $(20-r_1 ,a_1 ,20-r_2,a_2)$ for $\fp=2$ are topological mirrors, i.e. 
$$h^{p,q}(X)=h^{4-p,q}(\check{X})$$
\end{proposition}

Moreover, from the Euler characteristic formulas for the Borcea-Voisin orbifolds, one deduces immediately that the above does not hold anymore for $\fp$ 
odd\footnote{For $\fp=3$, one notices however that if two threefolds (resp. fourfolds) are based on the data $(r,a)$ and $(12-r,a)$ 
(resp. $(r_1,a_1,r_2,a_2)$ and $(12-r_1,a_1,12-r_2,a_2)$), their Euler characteristics are opposite (resp. identical).}.
 Indeed, it is easy to check that there is no integral constant $m$ such that the map $(r_1,r_2)\mapsto (m-r_1,m-r_2)$ leaves the Euler characteristic invariant.
 Therefore, we can abandon any hope of finding a mirror map between Borcea-Voisin orbifolds for $\fp>2$ in dimension 4.
Similarly, one checks readily that for  $\fp=3$, Borcea-Voisin threefolds do not come in mirror pairs; there is no involution on the set of pairs $(r,a)$ which sends a Hodge diamond to its reflection.\\

The non-closedness of the Borcea-Voisin family under mirror symmetry has several possible explanations. We enumerate here the three main obstacles.

\begin{enumerate}
 \item For $\fp>5$, there is no more symmetry of the type $r \leftrightarrow 20 - r$ between the possible pairs $(r,a)$ defining K3 surfaces with an order $\fp$ non-symplectic automorphism. See e.g. Figure \ref{fig:cas1}. For $\fp=3$ or $5$, one can find an axis of symmetry (up to a few elements) among the pairs $(r,a)$ but one can check that there is no topological relation binding the orbifolds which are based on symmetric pairs.

 \item Mirror symmetry is predicted to exist for Calabi-Yau manifolds near a \emph{large complex structure point}. For the threefold case, $\fp=3$, the elliptic curve is rigid. It is the quotient of the complex plane by the lattice generated by $1$ and a primitive sixth root of unity. One can not endow it with a large complex structure. For $\fp=2$, the involution exists for any complex curve of genus 1 and there is thus no obstruction to deforming the complex structure. 

 \item For the fourfolds, except when $\fp=2$, the fixed locus contains isolated fixed points and there is, for that reason, no guaranteed crepant resolution. Actually, we can show that most of the orbifolds which we construct have no crepant resolution:

\begin{proposition}
Let $\fp>2$ and $X$ be the Borcea-Voisin orbifold defined by $(r_1,a_1,r_2,a_2)$, then $X$ does not have a crepant resolution except when $\fp=3$ and $r_1=r_2=2$.
\end{proposition}

\begin{proof}
 If X has a fixed point of type $\frac1\fp (2,\fp-1,1,\fp-2)$, we know from Batyrev-Dais \cite{BD96} and Reid \cite{Rei02} that there does not exist a crepant resolution. Such points exist as soon as both surfaces involved in the construction have a fixed point of type $\frac1\fp (2,\fp-1)$. Indeed, the product point will have a linearized action under the product element in G of $(2,\fp-1)\times(2,\fp-1)^{-1}=(2,\fp-1,\fp-2,1)$. \\
From Table \ref{t:ast}, we see that for $\fp\in \left\{ 3,5, 7, 11\right\}$, any K3 surface admitting a non-symplectic automorphism of order $\fp$ has a fixed point of type  $\frac1p (2,\fp-1)$ except for the surfaces where $\fp=3$ and $r=2$. In the latter case, there are only fixed curves on each K3 surface and the associated Borcea-Voisin orbifolds admit a crepant resolution.
\end{proof}

The above indices lead us to the belief that the mirrors of  generalized Borcea-Voisin orbifolds with $\fp\geq 3$ will not come under this form. It would be interesting to see if they could be englobed into a larger construction.

\end{enumerate}

\newpage

\appendix

\section{Possible non-symplectic actions on K3 surfaces}
\label{app:cas}

We enumerate hereunder (\ref{ss:cas}) all possible primes $\fp$ for which there exists a non-symplectic $\IZ/\fp\IZ$ action on a K3 surface. Also, we enumerate all pairs $(r,a)$ which determine the lattice of the K3 surfaces admitting such an action. The integer $r$ gives the rank of the invariant part $S$ of the cohomology under the group action and $a$ is defined as $\det(S)=\fp^a$. For $\fp=2$, a third invariant $\delta\in \{0,1\}$ is needed. The elements corresponding to $\delta=0$ have been marked with a $*$ in the list.

The importance of $r$ and $a$ is that these two natural numbers completely determine the fixed locus of the group action. Theorem \ref{thm:cas} hereunder, due to Artebani, Sarti and Taki, gives an overview of the state of the art. It encompasses the classification of involutions, done by Nikulin \cite{Ni81}, of automorphisms of order 3, done by Artebani-Sarti \cite{AS08} and Taki \cite{Ta08}, of automorphisms of order 5 and 7, done by Artebani-Sarti-Taki \cite{AST09}, of automorphisms of order 11, done by Oguiso-Zhang \cite{OZ99}, and of automorphisms of order 13, 17 and 19 done by Vorontsov \cite{Vo83}, Kond\={o} \cite{Kon92} and by Oguiso-Zhang \cite{OZ00}.

\begin{theorem}\cite[Theorem 0.1]{AST09}
\label{thm:cas} Let $S$ be a hyperbolic $\fp$-elementary lattice ($\fp$ prime) of rank r with $\det(S)=\fp^a$. Then S is isometric to the invariant lattice of a non-symplectic automorphism $\sigma$ of order $\fp$ on a K3 surface if and only if $$ 22-r-(\fp-1)a \in 2 (\fp-1) \IZ_{\geq 0}$$
Moreover, if $\sigma$ is such automorphism, then its fixed locus $X^\sigma$ is the disjoint union of smooth curves and isolated points, and has the following form:
$$X^\sigma= \begin{cases}
             \emptyset & \mathrm{if } S\cong U(2)\oplus E_8(2) \\
	     E_1 \cup E_2 & \mathrm{if } S\cong U\oplus E_8(2) \\
	     C \cup R_1 \cup \ldots \cup R_k \cup \{p_1,\ldots,p_n\} & \mathrm{otherwise}
            \end{cases}$$
where $E_i$ is a smooth elliptic curve, $R_i$ is a smooth rational curve, $p_i$ is a an isolated point, $C$ is a curve of genus 
\begin{equation}\label{eq:genus}
g=\frac{22-r-(\fp-1)a}{2(\fp-1)}
\end{equation}
and
\begin{center} 
\begin{table}[ht]
 \begin{tabular}{ccccccc}
\hline\\[-0.9em]
$\fp$ & $2$ & $3,5,7$ & $11$ & $13$ & $17$ & $19$ \\
\hline 
\hline\\[-0.7em]
$n$ & $0$ & $\frac{-2+(\fp-2)r}{\fp-1}$  & $\frac{2+9r}{10}$ & $9$ & $7$ & $5$\\ [.3em]
$k$ &  $\frac{r-a}{2}$ & $\frac{2+r-(\fp-1)a}{2(\fp-1)}$ & $\frac{-2+r-10a}{20}$  & $1$ & $0$ & $0$\\[.3em]
\hline
\end{tabular}
\vspace{0.5cm}
\caption{Number of fixed points and rational curves}
\label{t:ast2}
\end{table}
\end{center}
\vspace{-1cm}

with the convention that $X^\sigma$ contains no fixed curves if $k=-1$.
\end{theorem}

To determine the topological type of our Borcea-Voisin orbifolds, we also need to characterize the fixed points of each action. Namely, given a prime $\fp$, an isolated fixed point of a non-symplectic automorphism will be called of type $\frac1\fp (i+1,\fp-i)$($i \in \{0,\ldots,\fp-2\}$), if the action can locally be linearized as 
$$\begin{pmatrix}
\zeta_\fp^{i+1} & 0\\
0 & \zeta_\fp^{\fp-1}
  \end{pmatrix}$$
We will denote by $n_i$ the number of such points. Artebani, Sarti and Taki \cite{AST09} give us a complete list of the number of such points in Table \ref{t:ast}.

\begin{center} 
\begin{table}[ht]
 \begin{tabular}{cccccccccccc}
\hline\\[-0.9em]
${p}$ & ${\alpha}$ & ${n_{1}}$ & ${n_{2}}$ &${n_{3}}$ &${n_{4}}$ &${n_{5}}$ &${n_{6}}$ &${n_{7}}$ &${n_{8}}$ & ${n_{9}}$ & ${p}$   \\ 
\hline \hline\\[-0.7em]
${2}$ & ${r-10}$ & & & & & & & & & & ${0}$ \\ 
${3}$ & $\frac{r-8}{2}$ &  ${\alpha+3}$ & & & & & & & & & ${\alpha+3}$   \\ 
${5}$ & $\frac{r-6}{4} $ &   ${2\alpha+3}$ & ${ 1+\alpha}$   && &&&&&  &   ${3\alpha+4}$    \\ 
${7}$ & $ \frac{r-4}{6} $ & ${2\alpha+2}$ & ${ 1+2\alpha}$ & ${ \alpha}$ &&&&&& & ${5\alpha+3}$ \\ 
${11}$ & $\frac{r-2}{10} $   & ${1+2\alpha}$ & ${2\alpha}$ & ${ 2\alpha}$ & ${ 1+2\alpha}$& ${\alpha} $   & & & &   & ${9\alpha+2}$ \\
${13}$ & $\frac{r+2}{12}$ & ${1+2\alpha}$ & ${1+2\alpha}$ & ${ 2\alpha}$ & ${ 2\alpha-1}$ & ${2\alpha-2} $& ${\alpha-1}$ & & & & ${11\alpha-2}$  \\
${17}$ & $\frac{r-6}{16}$ & ${2\alpha}$& ${2\alpha}$& ${2\alpha}$& ${2\alpha}$& ${2\alpha+1}$ & ${2\alpha+2}$ & ${ 2\alpha+3}$ & ${ \alpha+1}$ &  & ${15\alpha+7}$  \\ 
${19}$ & $\frac{r-4}{18}$ & ${2\alpha}$ & ${ 2\alpha} $ & $ {2\alpha} $ & $ {2\alpha+1} $ & $ {2\alpha+2} $ & $ {2\alpha+1} $ & $ {2\alpha+1} $ & ${ 2\alpha} $ & ${ \alpha}$ & ${17\alpha+5}$  \\ 
\hline
\end{tabular}
\vspace{0.5cm}
\caption{Types of isolated fixed points}
\label{t:ast}
\end{table}
\end{center}
\vspace{-.5cm}

\subsection{Enumeration of cases}
\label{ss:cas}

\subsubsection{Order 2} \textit{(64 cases)} (2,0)* (2,2)* (3,1) (3,3) (4,2) (4,4) (5,3) (5,5) (6,2)* (6,4)* (6,6) (7,3) (7,5) (7,7) (8,2) (8,4) (8,6) (8,8) (9,1) (9,3) (9,5) (9,7) (9,9) (10,0)* (10,2)* (10,4)* (10,6)* (10,8)*\footnote{Fixed locus consists of two disjoint genus 1 curves.} (10,10)*\footnote{Fixed locus is empty.} (11,1) (11,3) (11,5) (11,7) (11,9) (11,11) (12,2) (12,4) (12,6) (12,8) (12,10) (13,3) (13,5) (13,7) (13,9) (14,2)* (14,4)* (14,6)* (14,8) (15,3) (15,5) (15,7) (16,2) (16,4) (16,6) (17,1) (17,3) (17,5) (18,0)* (18,2)* (18,4)* (19,1) (19,3) (20,2)
\subsubsection{Order 3} \textit{(24 cases)} (2,0) (2,4) (4,1) (4,3) (6,2) (6,4) (8,1) (8,3) (8,5) (8,7) (10,0) (10,2) (10,4) (10,6) (12,1) (12,3) (12,5) (14,2) (14,4) (16,1) (16,3) (18,0) (18,2) (20,1)
\subsubsection{Order 5} \textit{(7 cases)} (2,1) (6,2) (6,4) (10,1) (10,3) (14,2) (18,1)
\subsubsection{Order 7} \textit{(5 cases)} (4,1) (4,3) (10,0) (10,2) (16,1) 
\subsubsection{Order 11} \textit{(3 cases)} (2,0) (2,2) (12,1)
\subsubsection{Order 13} \textit{(1 case)} (10,1)
\subsubsection{Order 17} \textit{(1 case)} (6,1)
\subsubsection{Order 19} \textit{(1 case)} (4,1)

\vspace{0.5cm}

Alternatively, all cases are displayed on Figure \ref{fig:cas1} for $\fp=2$ and Figure \ref{fig:cas2} for $\fp>2$.

\begin{figure}[htp]
	\psfrag{a}{$\scriptstyle{a}$}
	\psfrag{r}{$\scriptstyle{r}$}
	\psfrag{v0}{$\scriptstyle{0}$}
	\psfrag{v10}{$\scriptstyle{10}$}
	\psfrag{v20}{$\scriptstyle{20}$}
	\psfrag{v1}{$\scriptstyle{1}$}
	\psfrag{d0}{$\scriptstyle{\delta=0}$}	
	\psfrag{d1}{$\scriptstyle{\delta=1}$}
	\centering
	\includegraphics[height=4cm]{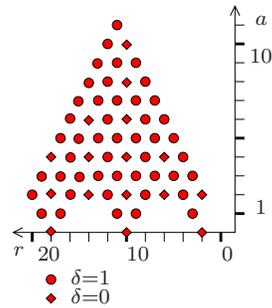}\\
	\caption{All possible pairs $(r,a)$ for $\fp=2$}
	\label{fig:cas1}
\end{figure}

\begin{figure}[htp]
	\psfrag{a}{$\scriptstyle{a}$}
	\psfrag{r}{$\scriptstyle{r}$}
	\psfrag{2}{$\scriptstyle{p=3}$}
	\psfrag{3}{$\scriptstyle{p=5}$}
	\psfrag{4}{$\scriptstyle{p=7}$}
	\psfrag{5}{$\scriptstyle{p=13}$}
	\psfrag{6}{$\scriptstyle{p=17}$}
	\psfrag{7}{$\scriptstyle{p=11}$}
	\psfrag{8}{$\scriptstyle{p=19}$}
	\psfrag{v0}{$\scriptstyle{0}$}
	\psfrag{v10}{$\scriptstyle{10}$}
	\psfrag{v20}{$\scriptstyle{18}$}
	\psfrag{v1}{$\scriptstyle{1}$}
	\psfrag{v5}{$\scriptstyle{5}$}
	\centering
	\includegraphics[height=4cm]{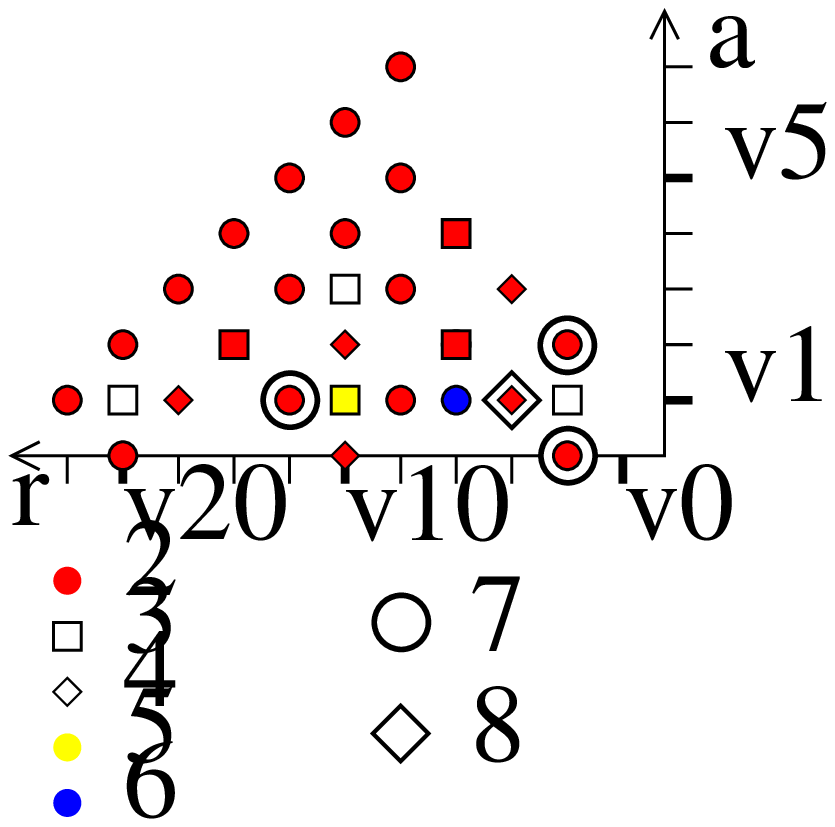}\\
	\caption{All possible pairs $(r,a)$ for $\fp>2$}
	\label{fig:cas2}
\end{figure}


\bibliographystyle{hplain}
\bibliography{sources}

\end{document}